\documentclass[11pt]{amsart}
\usepackage{amsfonts,amssymb,amscd}
\usepackage{oldgerm}
\usepackage{amsfonts}
\usepackage{amssymb}
\usepackage{graphics}

\newcommand{\mc}{\mathbb{C}}
\newcommand{\C}{\mathbb{C}}
\newcommand{\N}{\mathbb{N}}

\newcommand{\fol}{\mathcal{F}}

\newcommand{\cald}{{\mathcal{D}}}

\def\picill#1by#2(#3)#4
{\vbox to #2
{\hrule width #1 height 0pt depth 0pt
\vfill\special{illustration #3 scaled #4}}}

\newtheorem{teo}{Theorem}[section]

\newtheorem{lemma}[teo]{Lemma}

\newtheorem{defnc}[teo]{Definition}

\begin{document}

\title[Separatrices and vector fields on $(\C^3,0)$]{$2$-dimensional Lie algebras and separatrices for vector fields on $(\C^3,0)$}

\author{Julio C. Rebelo \, \, \, \& \, \, \, Helena Reis}
\address{}
\thanks{}

\begin{abstract}
We show that holomorphic vector fields on $(\C^3,0)$ have separatrices provided that they are embedded in a rank~$2$
representation of a two-dimensional Lie algebra. In turn, this result enables us to show that
the second jet of a holomorphic vector field defined on a compact
complex manifold $M$ of dimension~$3$ cannot vanish at an isolated singular point provided that $M$ carries
more than a single holomorphic vector field.
\end{abstract}

\maketitle

\section{Introduction}

Consider a compact complex manifold $M$ of dimension~$3$ and denote by ${\rm Aut}\, (M)$
the group of holomorphic diffeomorphisms of $M$. It is well-known that ${\rm Aut}\, (M)$ is a
finite dimensional complex Lie group whose Lie algebra can be identified with $\mathfrak{X}\, (M)$, the space
of all holomorphic vector fields defined on $M$; see \cite{akhiezer}. In this paper, the following will be proved:

\vspace{0.2cm}

\noindent {\bf Theorem~A}. {\sl Consider a compact complex manifold $M$ of dimension~$3$ and assume that
the dimension of ${\rm Aut}\, (M)$ is at least~$2$. Let $Z$ be an element of $\mathfrak{X}\, (M)$ and suppose that
$p \in M$ is an isolated singularity of~$Z$. Then
$$
J^2 (Z) \, (p) \neq 0
$$
i.e., the second jet of $Z$ at the point $p$ does not vanish.}

\bigskip

Theorem~A is the most elaborate result in this paper. In terms of importance, however, Theorem~B below is a general local
result playing the central role in the proof of Theorem~A.

To state Theorem~B, we consider two holomorphic vector fields $X$ and $Y$ defined on a neighborhood $U$ of $(0,0,0) \in \C^3$ and
we assume that they
are not linearly dependent at every point of $U$. In other words, the distribution spanned by $X$ and $Y$ has rank~$2$
away from a proper analytic subset of $U$. We also suppose that $X$ and $Y$ generate a Lie algebra of dimension~$2$.
In other words, either
these vector fields commute or they satisfy the more general relation
$[X,Y]=c \, Y$, where $c \in \C^{\ast}$ and where the brackets stand for the commutator of two vector fields. Then we have:

\vspace{0.2cm}

\noindent {\bf Theorem~B}. {\sl Let $X$ and $Y$ be two holomorphic vector fields defined on a neighborhood $U$
of $(0,0,0) \in \C^3$ which are not linearly dependent on all of $U$. Suppose that $X$ and $Y$ vanish at the origin and that
one of the following conditions holds:
\begin{itemize}
  \item $[X,Y]=0$;
  \item $[X,Y]=c \, Y$, for some $c \in \C^{\ast}$.
\end{itemize}
Then there exists a germ of analytic curve $\mathcal{C} \subset \C^3$ passing through the origin and simultaneously
invariant under $X$ and $Y$.}

\bigskip

Having stated the main results obtained in this paper, it is convenient to place them in perspective with
respect to previous works. We shall begin with Theorem~B since it is essential novel ingredient leading
to the proof of Theorem~A.
For this, recall that a
singular holomorphic foliation of {\it dimension~$1$}\, on $(\C^n,0)$ is nothing but the foliation induced by the
local orbits of a holomorphic vector field having a singular set of codimension at least~$2$.
In the case of foliations having dimension~$1$, a {\it separatrix}\, is a {\it germ of}\, analytic curve passing
through the singular point and invariant under the foliation in question. In dimension~$2$, a remarkable theorem due to
Camacho and Sad \cite{casad} asserts that every holomorphic foliation on $(\C^2,0)$ possesses a separatrix; their paper then
completes a classical work by Briot and Bouquet.
Unfortunately, the existence of separatrices is no longer a general phenomenon once the dimension increases as shown
by Gomez-Mont and Luengo in \cite{gmont}. The paper \cite{gmont} also contains examples of foliations without separatrix
on a {\it singular surface}, an issue previously discussed in \cite{camacho2}. The examples provided in
\cite{gmont}, however, include singular surfaces realized as hypersurfaces of $\C^3$ and foliations realized by holomorphic vector fields.
A basic question motivated by Gomez-Mont and Luengo's examples and aiming at finding appropriate generalizations
of Camacho-Sad theorem concerns the existence of separatrices for vector fields
as those considered in Theorem~B. This type of question also appears in the work of Stolovitch, Vey, and Zung about
normal forms for abelian actions and invariant sets; see the survey \cite{stolovitchsurvey} and its reference list.
From this point of view, Theorem~B is satisfactory for $(\C^3,0)$. Finally, in the commutative case,
Theorem~B nicely complements the main result in \cite{rebeloreis} concerning the existence of separatrices for the
{\it codimension~$1$ foliation}\, spanned by~$X$ and~$Y$. The reader will also note that the analogue of \cite{rebeloreis}
in the case of affine actions is known to be false since the classical work of Jouanolou; see \cite{Joa}.

Concerning Theorem~A, it essentially constitutes
a partial answer to a question raised by Ghys long ago. This question is better formulated in the context
of {\it semi-complete}\, vector fields; see \cite{JR1} or Section~5. In fact, Ghys has asked whether or not
an isolated singular point $p$ of a semi-complete vector field $Z$ always satisfies $J^2 (Z) (p) \neq 0$.
The answer is known to be affirmative in dimension~$2$. Whereas no counterexample is known in higher dimension,
it appears to exist a consensus that an affirmative answer to Ghys question, say in dimension~$3$, cannot be
obtained without a comprehensive and long analysis involving, in particular, reduction of singularities theorems which are themselves
fairly complicated already in dimension~$3$. From this point of view, the
advantage of Theorem~A lies in the fact that it provides a shortcut to an affirmative statement, albeit this statement
is slightly weaker than the original conjecture. Nonetheless, Ghys question was motivated by the potential applications of
this type of result to problems about bounds for the dimension
of the automorphism group of compact complex manifolds. As far as this type of application is targeted, the reader will note that
Theorem~A is satisfactory since the additional assumption on which it relies can be assumed to hold without
loss of generality.

The discussion conducted in the paper is rather elementary and relies on well-known results. Besides fairly
standard facts concerning singular spaces, we also use Malgrange's celebrated theorem in \cite{malgrange}, a
recent work by Guillot \cite{adolfo} concerning certain ``singular'' Kato surfaces equipped with vector fields, some
basic knowledge of complex surfaces, and Milnor's fibration theorem along with some related material.

\section{Codimension~$1$ foliations and invariant curves}

Recall that a singular {\it holomorphic foliation of dimension~$1$}\, on $(\C^3,0)$ is, by definition, given by
the local orbits of a holomorphic vector field $Z$ whose singular set ${\rm Sing}\, (Z)$ has codimension~$2$ or greater.
Similarly, a {\it codimension~$1$ holomorphic foliation}\, on $(\C^3,0)$ is associated
with the distribution obtained through the kernel of a holomorphic $1$-form $\Omega =\alpha \, dx + \beta \, dy + \gamma \, dz$ satisfying
the Frobenius condition $\Omega \wedge d\Omega =0$ and having a
singular set ${\rm Sing}\, (\Omega)$ of codimension at least~$2$. In terms of notation,
foliations of dimension~$1$ will typically be denoted by ``$\fol$'' whereas ``$\cald$'' will stand for codimension~$1$ foliations.

Henceforth we shall consider the setting of Theorem~B. Hence there are holomorphic vector fields
$X$ and $Y$ defined on a neighborhood $U$ of $(0,0,0) \in \C^3$ which are not parallel at every point $p$ in $U$. These
vector fields are assumed either to commute or to satisfy the equation $[X,Y]=cY$, for some $c \in \C^{\ast}$ where $[X,Y]$
stands for the commutator of the two vector fields in question.
In particular, $X$ and $Y$ span a singular codimension~$1$ foliation denoted by $\cald$ whose singular set ${\rm Sing}\, (\cald)$ has
codimension at least~$2$.

To begin our approach to Theorem~B, note that ${\rm Sing}\, (\cald)$ is clearly invariant under both $X$ and $Y$. Thus
the statement of Theorem~B is immediately verified provided that
the analytic set ${\rm Sing}\, (\cald)$ has dimension~$1$. Therefore, throughout this section, we shall assume that
${\rm Sing}\, (\cald)$ is reduced to the origin if not empty. In other
words, the foliation $\cald$ is either regular or it has an isolated singular point at the origin. Now Malgrange's theorem
in \cite{malgrange} ensures that $\cald$ is given by the level surfaces of some holomorphic function $f : (\C^3, 0) \rightarrow
(\C,0)$. Summarizing, in order to prove Theorem~B, we can assume without loss of generality that the
following lemma holds:

\begin{lemma}
\label{firstreductionsTheoremB-1}
The codimension~$1$ foliation $\cald$ admits a non-constant holomorphic first integral $f : (\C^3, 0) \rightarrow (\C,0)$
(which is, in fact, a submersion) so that the leaves of $\cald$ coincide with the level sets of $f$. Furthermore,
the foliation $\cald$ is either regular or it has an isolated singular point at the origin.\qed
\end{lemma}

To abridge notation, an analytic set of dimension~$2$ will be called an analytic surface (or simply a surface if no misunderstood
is possible). Similarly an analytic set of dimension~$1$ will be called an analytic curve (or simply curve).

Next consider the germ of analytic surface $S$ given by $S =f^{-1} (0)$ which is clearly invariant
under both~$X$ and~$Y$. This surface can be assumed to be {\it irreducible}\, otherwise two irreducible components
of it would intersect each other over a curve invariant under~$X$ and~$Y$ and hence satisfying the requirements
of Theorem~B. In any event, the subsequent discussion makes sense for every irreducible component $S$ of $f^{-1} (0)$.
In the sequel, the restrictions of~$X$ and~$Y$ to $S$ will be respectively denoted by $X_{\vert S}$ and by $Y_{\vert S}$.
Now, we have:

\begin{lemma}
\label{firstreductionsTheoremB-2}
Suppose that the restrictions $X_{\vert S}$ and $Y_{\vert S}$ of $X$ and $Y$ to $S$ are not everywhere parallel.
Then there exists a germ of analytic curve $\mathcal{C}\subset S$ invariant under both $X$ and $Y$.
\end{lemma}

\begin{proof}
Consider first the {\it singular set ${\rm Sing}\, (S) \subset \C^3$ of $S$ viewed as an analytic surface in $\C^3$}.
Apart from being possibly empty, ${\rm Sing}\, (S)$ is an analytic set whose dimension does not exceed~$1$.
Furthermore ${\rm Sing}\, (S)$
is naturally invariant under both $X$ and $Y$. Hence, if the dimension of ${\rm Sing}\, (S)$ equals~$1$, the statement
of the lemma follows at once.

Therefore, to prove the lemma, we can assume that either ${\rm Sing}\, (S)$ is empty or it is reduced to the origin. In either case,
let us consider the set ${\rm Tang}\, (X_{\vert S}, Y_{\vert S})$ consisting of those points $q \in S$ such that
$X(q)$ and $Y(q)$ are linearly dependent. By assumption, ${\rm Tang}\, (X_{\vert S}, Y_{\vert S})$ contains the origin so that
it is not empty. Similarly, this set is a proper analytic subset of $S$.

Moreover ${\rm Tang}\, (X_{\vert S}, Y_{\vert S})$ is invariant under both $X$ and $Y$
since $X$ and $Y$ generates a Lie algebra of dimension~$2$.
Therefore, if this set happens to be an analytic curve, the statement of the lemma results at once.
Thus, summarizing what precedes, we may assume the following holds:
\begin{itemize}
  \item The surface $S$ is either smooth or it has a unique singular point at the origin;

  \item The {\it tangency set}\, ${\rm Tang}\, (X_{\vert S}, Y_{\vert S})$ is reduced to the origin.
\end{itemize}
To complete the proof of the lemma, it suffices to check that the above conditions cannot simultaneously be satisfied.
For this, note that these conditions imply that the tangent sheaf to~$S$ is
locally free. Owing to the main result in \cite{scheja}, it follows that $S$ is smooth. However, since $S$ is smooth,
the vector fields $X_{\vert S}$ and $Y_{\vert S}$ can be identified with vector fields defined around the origin of $\C^2$
and linearly dependent at the origin. Clearly, these two vector fields must remain linearly dependent over some analytic curve
unless they are linearly dependent on a full neighborhood of $(0,0) \in \C^2$. The lemma is proved.
\end{proof}

In view of the preceding, in order to prove Theorem~B, we can assume without loss of generality that the vector fields
$X$ and $Y$ are parallel at every point of $S$. In other words, the vector fields $X_{\vert S}$ and $Y_{\vert S}$ are everywhere
parallel and thus they induce a unique singular holomorphic foliation on $S$ which will be denoted by
$\fol_S$. Theorem~B is now reduced to showing that the foliation $\fol_S$ defined on the analytic surface
$S$ admits a separatrix through the origin. Even though this is not necessary, we may then assume that $S$ is not smooth,
otherwise the existence of the mentioned separatrix is provided by the main result of \cite{casad}. In particular,
Theorem~B holds if the codimension~$1$ foliation $\cald$ is, indeed, non-singular. In fact, in more accurate terms,
Theorem~B is now a consequence of the following result:

\begin{teo}
\label{firstreductionsTheoremB-3}
Suppose that $X$ and $Y$ are as in Theorem~B and that they span a codimension~$1$ foliation $\cald$ having an isolated
singularity at $(0,0,0) \in \C^3$. Let $S$ be an analytic surface containing the origin and invariant under $\cald$
and assume that neither $X$ nor $Y$ vanishes identically on~$S$.
Finally, assume also that $X$ and $Y$ are everywhere parallel on $S$. Then there is a germ of analytic
curve $\mathcal{C}$ passing through the origin which is invariant under~$X$ and~$Y$.
\end{teo}

Theorem~\ref{firstreductionsTheoremB-3} is eminently a two-dimensional result valid for singular spaces; its proof
begins by observing that the origin must be the unique zero of both $X$ and $Y$ on $S$, otherwise the zero-set of
$X$ is invariant by $Y$ (and conversely) which immediately yields the desired invariant curve. Another useful observation
is provided by Lemma~\ref{firstreductionsTheoremB-4} below which is stated in a slightly more general setting. Let
$\overline{S} \subset \C^3$ be a germ of analytic surface with an isolated singular point at $(0,0,0) \in \C^3$.
Given a singular holomorphic foliation $\fol$ on $\overline{S}$, by a {\it holomorphic first integral}\, $h$ for $\fol$
it is meant a holomorphic function defined on $\overline{S} \setminus \{ (0,0,0) \}$ which is constant on the leaves of $\fol$.
With this terminology, we state:

\begin{lemma}
\label{firstreductionsTheoremB-4}
Given $\overline{S}$ and $\fol$ as above, assume that $\fol$ admits a non-constant holomorphic first integral $h$.
Then $\fol$ has a separatrix.
\end{lemma}

\begin{proof}
The argument is standard. For terminology and further detail, the reader is referred to \cite{nishino}, pages 210-212.
First we claim that $h$ is {\it weakly holomorphic}\, on all of $\overline{S}$ meaning that $h$ is holomorphic
on the regular part of $\overline{S}$ and is bounded on a neighborhood of the isolated singular point of $\overline{S}$ (identified
with the origin). To check the claim, note that $h$ is assumed to be holomorphic, and therefore weakly holomorphic, on $\overline{S}
\setminus \{ (0,0,0) \}$. Since the origin has codimension~$2$ in $\overline{S}$, it follows that $h$ is, in fact,
weakly holomorphic on all of $\overline{S}$. On the other hand, $\overline{S}$ is also a normal analytic surface since $\overline{S}$
is contained in $\C^3$ and has isolated singular points. Now, the fact that $\overline{S}$ is normal ensures
that $h$ has a holomorphic extension $H$ defined on a neighborhood of $(0,0,0) \in \C^3$.

To finish the proof of the lemma consider the holomorphic function $H : (\C^3,0) \rightarrow \C$ and set $H(0,0,0) = \lambda$.
Next, note that the intersection between $\overline{S}$ and the analytic surface
$H^{-1} (\lambda)$ is a certain analytic curve $\mathcal{C}$ passing through the origin and contained in $S$. Since the restriction
$h$ of $H$ to $\overline{S}$ is a first integral for $\fol$, it follows that $\mathcal{C}$ is invariant by $\fol$ and
thus it constitutes a separatrix for this foliation.
\end{proof}

Given a non-identically zero holomorphic vector field defined on an open set $U$ of some complex manifold,
the {\it order of $Z$ at a point $p \in U$}, ${\rm ord}_p (Z)$, is simply
the order of the first non-identically zero jet of $Z$ at $p$. In particular ${\rm ord}_p (Z) = 0$ if and only
if $Z$ is regular at $p$.

Going back to the statement of Theorem~\ref{firstreductionsTheoremB-3}, let us now prove a rather useful lemma.

\begin{lemma}
\label{firstreductionsTheoremB-5}
Let $X$, $Y$, and $S$ be as in Theorem~\ref{firstreductionsTheoremB-3} and denote by
$\fol_S$ the foliation induced by~$X$ and~$Y$ on~$S$. Then the restrictions $X_{\vert S}$ and
$Y_{\vert S}$ of $X$ and $Y$ to $S$ must commute.
\end{lemma}

\begin{proof}
We assume aiming at a contradiction that $[X_{\vert S} , Y_{\vert S}] = c Y_{\vert S}$ for some $c \in \C^{\ast}$.
Recall also that $X_{\vert S}$ and $Y_{\vert S}$ have a unique zero which coincides with the singular point of $S$.
Since $X_{\vert S}$ and $Y_{\vert S}$ are everywhere parallel and $X_{\vert S}$ does not vanish on the regular part
of~$S$, we have $Y_{\vert S} = h X_{\vert S}$ for some holomorphic function $h$ defined on $S \setminus \{ (0,0,0) \}$.
The argument employed in the proof of Lemma~\ref{firstreductionsTheoremB-4} then implies that $h$ admits a holomorphic
extension $H$ defined on a neighborhood of the origin in $\C^3$. We can assume that $H(0,0,0) \neq 0$ since otherwise
the set $H^{-1} (0)$ would intersect $S$ on a analytic curve contained in the zero-set of $Y_{\vert S}$ which contradicts
the assumption that $Y_{\vert S}$ has isolated zeros.

In the sequel, we assume without loss of generality that $H(0,0,0) =1$. Consider a (minimal) resolution $\widetilde{S}$ of
$S \subset \C^3$. The vector fields $X_{\vert S}$ and $Y_{\vert S}$ admit holomorphic lifts $\widetilde{X}_{\vert S}$
and $\widetilde{Y}_{\vert S}$ to $\widetilde{S}$, see for example \cite{adolfo}.
Similarly, the function $H$ induces the holomorphic function $\widetilde{H}
: \widetilde{S} \rightarrow \C$ given by $\widetilde{H} = H \circ \Pi$ where $\Pi$ stands for the resolution map
$\Pi : \widetilde{S} \rightarrow S$. In particular, $\widetilde{H}$ is constant equal to~$1$ on the exceptional
divisor $\Pi^{-1} (0,0,0)$.

We claim that there is a point $p \in \Pi^{-1} (0,0,0)$ at which the order ${\rm ord}_p (\widetilde{X}_{\vert S})$
of $\widetilde{X}_{\vert S}$ is at least~$1$.
To check the claim, recall that ${\rm ord}_p (\widetilde{X}_{\vert S}) =0$ if and only if $\widetilde{X}_{\vert S}$ is regular
at $p$. Thus, if the claim is false, the vector field $\widetilde{X}_{\vert S}$ must be regular on all of the exceptional
divisor $\Pi^{-1} (0,0,0)$. Assume that this is the case and consider an irreducible component $D$ of $\Pi^{-1} (0,0,0)$.
The curve $D$ must be invariant under $X$ otherwise its self-intersection would be nonnegative what is impossible.
This means that $X$ is tangent to $D$ and, since $X$ has no singular points in $\Pi^{-1} (0,0,0)$, $D$ must be an
elliptic curve and regular leaf for the foliation associated to $X$. The index formula of \cite{casad}, or simply the standard
Bott connection, implies that the self-intersection of $D$ still vanishes and this yields the final contradiction.

Let then $p \in \Pi^{-1} (0,0,0)$ be such that ${\rm ord}_p (\widetilde{X}_{\vert S}) \geq 1$. Since $\widetilde{H} (p) =1$,
we have ${\rm ord}_p (\widetilde{X}_{\vert S}) = {\rm ord}_p (\widetilde{Y}_{\vert S})$. Now,
suppose first that ${\rm ord}_p (\widetilde{X}_{\vert S})
\geq 2$. Then the order ${\rm ord}_p ([\widetilde{X}_{\vert S} ,\widetilde{Y}_{\vert S}])$ of
$[\widetilde{X}_{\vert S} ,\widetilde{Y}_{\vert S}]$ at~$p$ satisfies
$$
{\rm ord}_p ( [\widetilde{X}_{\vert S} ,\widetilde{Y}_{\vert S}]) \geq 2 \, {\rm ord}_p (\widetilde{X}_{\vert S}) -1 >
{\rm ord}_p (\widetilde{X}_{\vert S}) = {\rm ord}_p (\widetilde{Y}_{\vert S})
$$
so that the equation $[\widetilde{X}_{\vert S} ,\widetilde{Y}_{\vert S}] = c \widetilde{Y}_{\vert S}$ with $c \in \C^{\ast}$
cannot hold. Therefore it only remains to check the possibility of having
${\rm ord}_p (\widetilde{X}_{\vert S}) =1$. For this, note again that $\widetilde{Y}_{\vert S} = \widetilde{H} \widetilde{X}_{\vert S}$
so that the first jet of $\widetilde{X}_{\vert S}$ and of $\widetilde{Y}_{\vert S}$ coincide at $p$ (recall that
$\widetilde{H} (p) =1$). Thus the linear part of the commutator $[\widetilde{X}_{\vert S} ,\widetilde{Y}_{\vert S}]$
must still vanish and this ensures that the order of $[\widetilde{X}_{\vert S} ,\widetilde{Y}_{\vert S}]$ at~$p$
is again strictly larger than the order of $\widetilde{Y}_{\vert S}$. This contradicts the equation
$[\widetilde{X}_{\vert S} ,\widetilde{Y}_{\vert S}] = c \widetilde{Y}_{\vert S}$ and
completes the proof of the lemma.
\end{proof}

The next lemma is an application of Lemmas~\ref{firstreductionsTheoremB-4} and~\ref{firstreductionsTheoremB-5}.

\begin{lemma}
\label{firstreductionsTheoremB-6}
Let $X$, $Y$, $S$, and $\fol_S$ be as in Theorem~\ref{firstreductionsTheoremB-3}. If the restrictions $X_{\vert S}$ and
$Y_{\vert S}$ of $X$ and $Y$ to $S$ do not differ by a multiplicative constant, then $\fol_S$ admits a separatrix.
\end{lemma}

\begin{proof}
As mentioned, the vector fields $X_{\vert S}$ and $Y_{\vert S}$ can be assumed to have an isolated zero at
$(0,0,0) \subset S$. Moreover Lemma~\ref{firstreductionsTheoremB-5} ensures that $[X_{\vert S} , Y_{\vert S}]=0$.
In particular, the function $h$ defined on $S \setminus \{ (0,0,0) \}$ by means of the equation
$Y_{\vert S} = h X_{\vert S}$ must be a first integral of~$\fol_S$. Hence
Lemma~\ref{firstreductionsTheoremB-4} ensures that $\fol_S$ admits a separatrix unless $h$ is constant. The lemma is proved.
\end{proof}

We are now ready to prove Theorem~\ref{firstreductionsTheoremB-3}.

\begin{proof}[Proof of Theorem~\ref{firstreductionsTheoremB-3}]
We assume for a contradiction that no germ of curve $\mathcal{C} \subset S$ passing through the origin
is simultaneously invariant under~$X$ and~$Y$. Next
consider the vector fields $X_{\vert S}$ and $Y_{\vert S}$ on $S$. As already mentioned,
the singular set of~$S$ must consist of isolated points. Moreover this singular set
contains the zero sets of both~$X$ and~$Y$. In particular,
Lemmas~\ref{firstreductionsTheoremB-5}
and~\ref{firstreductionsTheoremB-6} imply that the commutator $[X_{\vert S} , Y_{\vert S}]$ vanishes identically and
that a separatrix for $\fol_S$ exists unless $Y_{\vert S}$ is a constant multiple of $X_{\vert S}$. In the sequel, we
assume this to be the case.

Therefore consider $c_0 \in \C$ such that $Y_{\vert S} = c_0 X_{\vert S}$ and denote by $\mathfrak{G}$ the Lie algebra
generated by $X$ and~$Y$. Note that the vector field $Z_1 = Y - c_0 X$ clearly belongs to $\mathfrak{G}$ and, in
addition, vanishes identically over $S=f^{-1} (0)$, where $f$ is the first integral of $\cald$ mentioned in
Lemma~\ref{firstreductionsTheoremB-1}. Thus, there is $k_1 \in \N^{\ast}$ such that $Y_1 = Z_1 / f^{k_1}$
is a holomorphic vector field defined on a neighborhood of $(0,0,0) \in \C^3$ which does not vanish identically on~$S$
(recall that $S$ is assumed to be irreducible). Moreover, $Y_1$ is still tangent to the foliation
$\cald$ and hence it leaves $S$ invariant.

\noindent {\it Claim}. The restriction $Y_{1\vert S}$ of the vector field $Y_1$ to~$S$ is still a constant multiple
of $X_{\vert S}$.

\noindent {\it Proof of the Claim}. Since $f$ is a first integral of both $X$ and $Y$, the vector fields
$X$ and $Y_1$ either commute or generate a Lie algebra isomorphic to the Lie
algebra of the affine group. In other words, the Lie algebra
generated by $X$ and $Y_1$ is isomorphic to $\mathfrak{G}$. Therefore, if $Y_{1\vert S}$ were not a constant multiple
of $X_{\vert S}$, there would exist a germ of analytic curve contained in $S$, passing through the origin and
invariant under both $X$ and $Y_1$. This is impossible because such a curve is automatically invariant by $Y$ as well
since $Y_{\vert S} = c_0 X_{\vert S}$. The claim is proved.\qed

Setting $Y_{1\vert S} = c_1 X_{\vert S}$, it follows again that $Y_{1} - c_1 X = f^{k_2} Y_2$
where $Y_2$ is a holomorphic vector field defined around $(0,0,0) \in \C^3$ which is tangent to $S$ and does
not vanish identically on~$S$. Therefore, we have
$$
Y = c_0 X + f^{k_1} Y_1 = c_0 X + c_1 f^{k_1} X + f^{k_1+k_2} Y_2 \, .
$$
Again the fact that $f$ is a first integral for $X$ and $Y$ implies that $X$ and $Y_2$ generate a Lie algebra
isomorphic to $\mathfrak{G}$. Thus, up to repeating the above argument and assuming
that no germ of curve contained in $S$ and passing through $(0,0,0)$ is invariant by~$X$, we conclude that
$Y_2$ must have the form $Y_2 =c_2 X$ for some constant $c_2 \in \C$. The above procedure can then be repeated to yield
a holomorphic vector field $Y_3$ along with a certain $k_3 \in \N^{\ast}$ such that
$$
Y =  c_0 X + c_1 f^{k_1} X + c_2 f^{k_1+k_2} X + f^{k_1+k_2+k_3} Y_3 \, .
$$
Moreover, $Y_3$ is tangent to~$S$ and does not vanish identically on~$S$. Since $X$ and $Y$ are not parallel at every point in
a neighborhood of the origin in $\C^3$, by continuing this procedure we shall
eventually find a vector field $Y_n$ as above whose restriction to~$S$ will not be a constant multiple of~$X$.
A contradiction proving Theorem~\ref{firstreductionsTheoremB-3} will then immediately arise.
\end{proof}

\begin{proof}[Proof of Theorem B]
It follows from Theorem~\ref{firstreductionsTheoremB-3} combined to the discussion conducted in the
beginning of this section. We summarize the argument for the convenience of the reader.

Consider holomorphic vector fields $X$ and $Y$ as in Theorem~B. We assume aiming at a contradiction
that there is no (germ of) analytic curve $\mathcal{C} \subset \C^3$ containing the origin and simultaneously
invariant under $X$ and~$Y$. In particular, this implies that the codimension~$1$ foliation $\cald$ spanned
by $X$ and~$Y$ possesses an isolated singular point at the origin and it is given by the level surfaces of a
holomorphic function $f : (\C^3 ,0) \rightarrow (\C,0)$. Denote by $S$ an irreducible component of $f^{-1} (0)$
and consider the restrictions $X_{\vert S}$ and $Y_{\vert S}$ of $X$ and $Y$ to~$S$.

Since we are assuming that Theorem~B does not hold, it follows from Lemma~\ref{firstreductionsTheoremB-2}
that $X_{\vert S}$ and $Y_{\vert S}$ are parallel at every point of~$S$. Note that these two vector fields cannot
simultaneously be identically zero for otherwise the statement is obvious. If none of these vector fields vanish
identically then Theorem~B follows from Theorem~\ref{firstreductionsTheoremB-3}. Therefore there remains only
one possibility not yet covered by
our discussion which corresponds to the situation where one of the vector fields vanishes identically over~$S$
while the other has an isolated zero at the origin (identified with the isolated
singular point of~$S$). This case can however be dealt with by arguing as in the proof of Theorem~\ref{firstreductionsTheoremB-3}.
Namely, suppose for example that $Y_{\vert S}$ vanishes identically. Then, up to dividing $Y$ by a suitable power
of the first integral~$f$, we shall obtain a vector field $\overline{Y}$ whose restriction to~$S$ is no longer identically
zero. Again the Lie algebra generated by $X$ and by $\overline{Y}$ is isomorphic to the Lie algebra generated by $X$ and
by $Y$. Thus the procedure described in the proof of Theorem~\ref{firstreductionsTheoremB-3} to produce a separatrix
for the foliation induced on~$S$ can be started with the vector fields $X$ and $\overline{Y}$. This completes
the proof of Theorem~B.
\end{proof}

\section{Lie groups actions and holomorphic vector fields}

All vector fields considered in the remainder of this paper are assumed not to vanish identically unless otherwise stated.

Consider a compact complex manifold $M$ and let
${\rm Aut}\, (M)$ denote the group of holomorphic diffeomorphisms of $M$. It is well known that
${\rm Aut}\, (M)$ is a finite-dimensional complex Lie group whose action on $M$ is
faithful and holomorphic. Moreover, the Lie algebra of ${\rm Aut}\, (M)$ can be identified to the Lie algebra
$\mathfrak{X} \, (M)$ formed by all holomorphic vector fields defined on $M$, see for example \cite{akhiezer}.

Since $M$ is compact, every holomorphic vector field $X$ defined on all of $M$ is {\it complete}\, in the sense
that it gives rise to an action of $(\C,+)$ on $M$. In particular, the restriction of $X$ to every open set
$U \subset M$ is {\it semi-complete}\, on $U$, according to the definition given in \cite{JR1} which is recalled
below for the convenience of the reader.

\begin{defnc}
A holomorphic vector field~$X$ on a complex manifold~$N$ is called semi-complete
if for every~$p\in N$ there exists a connected domain~$U_p\subset\C$
with~$0\in U_p$ and a map~$\phi_p:U_p\to N$ satisfying the following conditions:
\begin{itemize}
\item $\phi_p(0)=p$ and $d \phi_p/d t|_{t=t_0}=X(\phi_p(t_0))$.
\item For every sequence~$\{t_i\}\subset U_p$ such
that~$\lim_{i\rightarrow\infty}t_i\in\partial U_p$ the sequence $\{\phi_p(t_i)\}$
escapes from every compact subset of~$N$.
\end{itemize}
\end{defnc}

If $X$ is semi-complete on $U$ and $V \subset U$ is another open set, then
the restriction of $X$ to $V$ is semi-complete as well. Hence, there is a well-defined notion of {\it semi-complete singularity}.
Furthermore, if a singularity of a vector field happens {\it not to be}\, semi-complete,
then this singularity cannot be realized as singularity of a vector field defined on a compact complex manifold.
Keeping this principle in mind, a fundamental result implicitly formulated in \cite{JR1} reads as follows.

\begin{lemma}
\label{semicomplete1.0}
Assume that $Y$ is a holomorphic vector field with isolated singular points and defined around
the origin of $\C^n$. Suppose that the second jet $J^2 (X) \, (0, \ldots ,0)$ of $X$ at the origin vanishes
and that $X$ is tangent to some analytic curve $\mathcal{C}$ passing through the origin. Then the germ
of $X$ is not semi-complete.
\end{lemma}

\begin{proof}
It suffices to sketch the elementary argument. The
restriction $X_{\mathcal{C}}$ of $X$ to $\mathcal{C}$ does not vanish identically since the origin is an
isolated singularity of $X$. Clearly it suffices to show that $X_{\mathcal{C}}$ is not semi-complete. For this,
recall that the curve $\mathcal{C}$ admits
an irreducible Puiseux parameterization $\mathcal{P} : (\C , 0) \rightarrow (\mathcal{C}, 0)$.
In particular, the pull-back $\mathcal{P}^{\ast} X_{\mathcal{C}}$ is a (non-identically
zero) holomorphic vector field defined on a neighborhood of $0 \in \C$. The proof of the lemma is then reduced
to checking that $\mathcal{P}^{\ast} X_{\mathcal{C}}$ is not semi-complete on any neighborhood of $0 \in \C$.

Since the second jet of $X$ vanishes at the origin of $\C^n$, it is straightforward to check that the second jet of
$\mathcal{P}^{\ast} X_{\mathcal{C}}$ vanishes at $0 \in \C$. Hence $\mathcal{P}^{\ast} X_{\mathcal{C}}$ is locally
given by~$z^{s}f(z)\partial/\partial z$ for some~$s\geq 3$ and a holomorphic function~$f$ satisfying $f(0) \neq 0$.
Now, assuming that $\mathcal{P}^{\ast} X_{\mathcal{C}} = z^{s}f(z)\partial/\partial z$ is semi-complete, it follows that
so are the vector fields $X_n =  \Lambda_n^{\ast} (z^{s}f(z)\partial/\partial z)$ where $\Lambda_n : \C \rightarrow \C$
is given by $\Lambda_n (z) = z/n$. Also every constant multiple of a semi-complete vector field is again semi-complete.
Putting together these two remarks, the
vector field $z^{s}\partial/\partial z$ is the uniform limit of a sequence of semi-complete vector fields which is
obtained by suitably renormalizing the sequence $\{ X_n \}$. There follows that $z^{s}\partial/\partial z$ must be semi-complete
as well, see for example \cite{ghysrebelo}. However the solution $\phi$ of the differential equation corresponding
to $z^{s}\partial/\partial z$ is given by
$$
\phi (T) = \frac{z_0}{\sqrt[s-1]{1 - Tz_0^{s-1}(s-1)}}
$$
where $\phi (0) = z_0$. Since $s \geq 3$, this solution is multivalued so that $z^{s}\partial/\partial z$ is never
semi-complete on a neighborhood of $0 \in \C$. The lemma follows.
\end{proof}

Let then $M$ be as in Theorem~A. In particular, the dimension of
$\mathfrak{X} \, (M)$ is at least~$2$. Hereafter, we assume aiming at a contradiction that $Y \in \mathfrak{X} \, (M)$
is a holomorphic vector field possessing an isolated singular point $p \in M$ where the second jet of $Y$ vanishes
(notation: $J^2 (Y) (p) = 0$).
The idea is then to exploit Theorem~B and Lemma~\ref{semicomplete1.0} to derive a contradiction implying Theorem~A. In this section,
we shall content ourselves with establishing a slightly weaker statement, namely:

\begin{teo}
\label{WeakTheoremA}
Assume $Y \in \mathfrak{X} \, (M)$ is a holomorphic vector field possessing an isolated singular point $p \in M$ where
its second jet vanishes. Then $Y$ admits a non-constant meromorphic first integral $f$ defined on~$M$. Moreover, the Lie
algebra $\mathfrak{X} \, (M)$ contains a vector field $X$ such that $X=fY$.
\end{teo}

In the next section, we shall study in detail the exceptional situation described in Theorem~\ref{WeakTheoremA} so as to exclude
its existence. Theorem~A will then automatically follow.

In the remainder of this section, we shall focus on the proof of Theorem~\ref{WeakTheoremA}.

For every non-identically zero vector field $Z \in \mathfrak{X} \, (M)$, denote by ${\rm ord}_p (Z)$ the order of $Z$ at
$p \in M$, see Section~2. We then set $k={\rm ord}_p (Y)$ so that, by assumption, we have $k \geq 3$.
Since the vector space $\mathfrak{X} \, (M)$ has finite
dimension, the subset of $\N$ formed by those integers that are realized as the order at $p$
of some vector field $Z \in \mathfrak{X} \, (M)$
is finite. The (attained) supremum of this set will be denoted by $m \in \N^{\ast}$. Clearly $m \geq k \geq 3$. Now let us
consider two different possibilities for $m$ and $k$.
\begin{itemize}
  \item[{\sc A}.] We have $m \geq k$ and, if $m=k$, then there is a vector field $X \in \mathfrak{X} \, (M)$
  which is not a constant multiple of $Y$ and satisfies ${\rm ord}_p (X) = m =k$.

  \item[{\sc B}.] We have $m=k$ and every vector field $Z \in \mathfrak{X} \, (M)$ satisfying
  ${\rm ord}_p (Z) = m =k$ is a constant multiple of~$Y$.
\end{itemize}

\noindent {\sc Case A}. To prove Theorem~\ref{WeakTheoremA} in this case, we proceed as follows. Fix a vector field
$X \in \mathfrak{X} \, (M)$, which is not a constant multiple of~$Y$, and whose order at $p$ equals~$m$. Note that
the commutator $[X,Y]$ must vanish identically since, otherwise, its order at $p$ is at least $m+k-1 \geq m+2 >m$
which contradicts the definition of~$m$.

\vspace{0.1cm}

\begin{proof}[Proof of Theorem~\ref{WeakTheoremA} in {\sc Case A}]
According to Lemma~\ref{semicomplete1.0}, the vector field $Y$ cannot admit a local separatrix through~$p$.
Since we have $[X,Y]=0$, Theorem~B implies that $X$ and $Y$ must be everywhere parallel so that they define a single
singular holomorphic foliation $\fol$ on $M$. Also we can write
$X = fY$ for a certain meromorphic function $f$ defined on~$M$. This meromorphic function is by assumption non-constant.
Furthermore, as already seen, the restriction of $f$ to a leaf of $\fol$ must
be constant since $[X,Y]=0$. The proof of Theorem~\ref{WeakTheoremA} in {\sc Case A} is then completed.
\end{proof}

\vspace{0.1cm}

\noindent {\sc Case B}. Consider a vector field $X \in \mathfrak{X} \, (M)$ which is not a constant multiple of
$Y$. The existence of $X$ is ensured by the fact that the dimension of $\mathfrak{X} \, (M)$ is at least~$2$.

\vspace{0.1cm}

\begin{proof}[Proof of Theorem~\ref{WeakTheoremA} in {\sc Case B}]
Consider the order ${\rm ord}_p (X)$ of $X$ at the point $p \in M$. Assume first that ${\rm ord}_p (X) \geq 2$.
Then $X$ and $Y$ must commute since otherwise the order of $[X,Y]$ is at least $k - 1 + {\rm ord}_p (X) \geq k+1$
which contradicts the fact that $m=k$. Since Lemma~\ref{semicomplete1.0} ensures that $Y$ cannot have a separatrix
through~$p$, it follows from Theorem~B that $X$ is everywhere parallel to~$Y$. Now the argument used to prove
Theorem~\ref{WeakTheoremA} in {\sc Case A} can also be applied to the present situation to complete the proof of the
theorem. In fact, the mentioned argument shows that whenever $\mathfrak{X} \, (M)$ contains a vector field commuting
with $Y$, the statement of Theorem~\ref{WeakTheoremA} must hold.

Suppose now that ${\rm ord}_p (X) \leq 1$. As mentioned above, we can assume without loss of generality that $X$ does not commute with~$Y$.
Consider then $Z = [X,Y] \in \mathfrak{X} \, (M)$. The order ${\rm ord}_p (Z)$ of $Z$ at $p$ is at least $k-1 \geq 2$
which, in turn, ensures that $[Z,Y]=0$. Otherwise the order of $[Z,Y]$ would be at least $k+k-1 -1 \geq k+1$ ($k \geq 3$)
which contradicts the assumption that $m$ is the greatest integer realized as the order at $p$ of a vector field in
$\mathfrak{X} \, (M)$. Since $[Z,Y]=0$, Theorem~~\ref{WeakTheoremA} will again follow unless $Z$ is a constant multiple of~$Y$. Hence,
we only need to discuss this last possibility.

In other words, we assume that $Z=[X,Y]=c Y$ for some $c \in \C^{\ast}$. To obtain a contradiction finishing the proof
of Theorem~\ref{WeakTheoremA}
it suffices to check that $X$ and $Y$ cannot be everywhere parallel. In fact, if $X$ and $Y$ are linearly independent
at generic points of $M$, then Theorem~B implies the existence of a separatrix for $Y$ through~$p$ which is impossible.
Finally, to check that $X$ and $Y$ are not everywhere parallel, suppose for a contradiction that they were parallel at every point
of $M$. Then we can again consider a meromorphic function $f$ such that $X = fY$. However this function $f$ must be
holomorphic around~$p$ since $p$ is an isolated singularity of $Y$ and $X$ is holomorphic.
From this it follows that the order of $X$ at $p$ is greater than
or equal to the order of $Y$ at $p$. A contradiction immediately arises since the former is bounded by~$1$ while the latter
is at least~$3$. The proof of the theorem is completed.
\end{proof}

\section{Proof of Theorem~A}

In this section we shall derive Theorem~A from Theorem~\ref{WeakTheoremA}. In order to do this, we need to study
the situation corresponding to the conclusion of Theorem~\ref{WeakTheoremA} so as to
rule out its existence. The discussion below
relies heavily on the recent work of Guillot \cite{adolfo}.

Therefore we suppose
that $M$ is a compact complex manifold of dimension~$3$ carrying a holomorphic vector field $Y$ which
has an isolated singularity $p \in M$ at which its second jet vanishes. In other words, we have $J^2 (Y) (p) =0$.
We also assume the existence
of another {\it holomorphic vector field}\, $X$ on $M$ having the form $X = fY$ where $f$ is a non-constant meromorphic
first integral of both $X$ and $Y$. Since $X$ is holomorphic, it follows that the divisor of poles of $f$ must be contained
in the divisor of zeros of $Y$ which is therefore non-empty. In any event, we know already that $f$ is holomorphic
on a neighborhood of $p$. Set $f(p) =\lambda \in \C$ and consider the compact (possibly singular) surface $S \subset M$ given
by $S = f^{-1} (\lambda)$. Denote by $Y_{\vert S}$ (resp. $X_{\vert S}$) the restriction of $Y$ (resp. $X$) to~$S$.

\begin{lemma}
\label{finishingjob-1.1}
The surface $S$ has a unique singular point at $p$ and its minimal resolution $\widetilde{S}$ is a Kato surface.
Furthermore the vector field $Y_{\vert S}$ has no singular point other than~$p \in M$.
\end{lemma}

\begin{proof}
Clearly $Y_{\vert S}$ is complete on~$S$ so that it is also semi-complete on a neighborhood of $p \in S$. Since
$p$ is an isolated singular point of $Y$ (and hence of $Y_{\vert S}$) where the second jet of $Y$ equals zero, it follows
from Lemma~\ref{semicomplete1.0} that the foliation on $S$ associated with $Y_{\vert S}$ has no separatrix at~$p$.
This is exactly the context of Guillot's work \cite{adolfo}. In fact, he shows that the minimal resolution of $S$ is
a Kato surface. Furthermore the Kato surface $\widetilde{S}$ carries a unique holomorphic vector field (represented
by the corresponding transform of $Y_{\vert S}$) and this vector field is regular away from the exceptional divisor. This implies
our lemma.
\end{proof}

In view of the preceding, we shall call $S$ {\it a singular Kato surface}. It is understood that a singular
Kato surface has a unique singular point and admits an ordinary Kato surface as minimal resolution.

Recall that the pole divisor of $f$ is contained in the zero-divisor of $Y$. Since $Y_{\vert S}$ has a unique
singular point (the isolated singularity~$p$), it follows that the zero-divisor of $Y$ does not intersect~$S$.
Therefore the meromorphic function $f$ is actually holomorphic on a neighborhood of $S$ in $M$. Thus, restricted
to a neighborhood $U$ of $S$, $f$ induces a structure of singular fibration in $U$. In fact, the following holds:

\begin{lemma}
\label{finishingjob-2.2}
The function $f$ induces
a proper holomorphic map $F : U \rightarrow \mathbb{D} \subset \C$, where $\mathbb{D}$ denote the unit disc of $\C$,
satisfying the following conditions:
\begin{enumerate}
  \item For every $z \in D \setminus \{ 0 \}$ the fiber $F^{-1} (z)$ over $z$ is a smooth complex surface.
  \item The fiber $F^{-1} (0)$ is a singular Kato surface whose singular point is denoted by $p \in U \subset M$.
  \item The vector field $Y$ is tangent to the fibers of $F$. Furthermore $p$ is the only singular point of $Y$ in $U$.
\end{enumerate}
\end{lemma}

\begin{proof}
To check the first item, note that this item only makes sense up to reducing $U$. Thus, if the statement were false,
the set of singular points of the fibers $F^{-1} (z)$ would give rise to an irreducible proper analytic subset of $M$
with dimension at least~$1$
intersecting $S = F^{-1} (0)$ at $p$ (since $F$ is a local submersion on $S$ away from~$p$). This analytic set
is clearly invariant by $Y$ so that it would yield a separatrix for $Y$ at $p$ provided that its dimension is equal to~$1$.
This being impossible, we only need to show that the set in question cannot have dimension~$2$ either. For this, just notice
that it would intersect $S$ on a curve which would consists of singular points of~$S$ whereas $S$ is regular away from~$p$.
The final contradiction establishes item~(1) of the statement. Item~(2) follows from Lemma~\ref{finishingjob-1.1}.

Naturally the statement in item~(3) is also to be understood up to reducing~$U$. To check it,
consider a neighborhood $V \subset U$ of $p$ containing no other singular point of $Y$. The existence of $V$ is nothing but
the assumption that $p$ is an isolated singular point for~$Y$. We also know that $Y_{\vert S}$ has no
singular point on the compact set $S \setminus V$. Thus there is a neighborhood $W \subset M$ of
$S \setminus V$ on which $Y$ is regular. It then suffices to choose $U$ contained in the union $V \cup W$.
The proof of the lemma is completed.
\end{proof}

The next step consists of further detailing the structure of the
fiber $F^{-1} (0)$. First recall that the general construction of Kato surfaces can be summarized as follows.
Consider a non-singular surface $\widehat{S}$ along with a divisor $\widehat{D}$ which can be collapsed to
yield a neighborhood of the origin in $\C^2$. In other words, $(\widehat{S} , \widehat{D})$ is obtained by means
of finitely many blow-ups sitting over the origin of $\C^2$. In particular, we have the contraction map $\widehat{\pi} :
(\widehat{S} , \widehat{D}) \rightarrow (\C^2,0)$. Next consider a point $q \in \widehat{D}$ which, for our purposes,
can be chosen as a regular point of $\widehat{D}$. Suppose we have a local holomorphic diffeomorphism
$\widehat{\sigma} : (\C^2,0) \rightarrow (\widehat{S} ,q)$. The pair $(\widehat{\pi}, \widehat{\sigma})$ is said to
be the {\it Kato data}\, of the surface. To obtain a Kato surface from the pair $(\widehat{\pi}, \widehat{\sigma})$,
we choose $\varepsilon >0$ sufficiently small and the manifold with boundary $\widehat{N}$ given by
$$
\widehat{N} = \widehat{\pi}^{-1} (B_{\varepsilon}^4 \cup \Delta_{\varepsilon}^3) \setminus \widehat{\sigma} (B_{\varepsilon}^4)
$$
where $B_{\varepsilon}^4 \subset \C^2$ stands for the open ball around the origin of radius~$\varepsilon$ and where
$\Delta_{\varepsilon}^3 = \partial B_{\varepsilon}^4$ is the boundary of $B_{\varepsilon}^4$. Naturally the superscripts~$3$ and~$4$
are intended to remind us of the corresponding real dimensions and they might be useful to avoid confusion in the subsequent discussion.
The boundary of $\widehat{N}$
has two connected components $\widehat{\pi}^{-1} (\Delta_{\varepsilon}^3)$ and $\widehat{\sigma} (\Delta_{\varepsilon}^3)$ which
can be identified to each other by means of $\widehat{\sigma} \circ \widehat{\pi}$. The quotient of this identification is
a {\it Kato surface}\, which can be made minimal up to contracting all exceptional curves.

Recall that $S =F^{-1} (0)$
has a unique singular point $p$ and it is equipped with a holomorphic vector field $Y_{\vert S}$ having no zeros
away from~$p$. Moreover, $S$ contains no germ of analytic curve passing through~$p$ and invariant under
$Y_{\vert S}$. Denoting by $\widetilde{S}$ the minimal (good) resolution of~$S$,
the main result of \cite{adolfo} can now be rephrased in our context as follows.
\begin{itemize}
  \item $\widetilde{S}$ is a Kato surface and comes equipped with a natural resolution map $\pi_S : \widetilde{S} \rightarrow S$
  which is a diffeomorphism away from $\widetilde{D} = \pi^{-1} (p) \subset \widetilde{S}$.

  \item The vector field $Y_{\vert S}$ lifts to a holomorphic vector field $\widetilde{Y}_{\vert S}$ on $\widetilde{S}$
  whose associated foliation is denoted by $\widetilde{\fol}_{\vert S}$ (note that $\widetilde{Y}_{\vert S}$ possesses a curve
  of zeros).

  \item The irreducible components of the divisor $\widetilde{D} = \pi^{-1} (p)$ are all rational curves and they are invariant
  under $\widetilde{\fol}_{\vert S}$.

  \item The dual graph of $\widetilde{D} = \pi^{-1} (p)$ contains a unique cycle (and it is not reduced to this cycle).
  Furthermore, the only vertices of this graph having valence greater than~$2$ belong to the cycle and their valences are
  exactly~$3$ (see figure \ref{uniquefigure}).
\end{itemize}
As mentioned, we shall say that $S = F^{-1} (0)$ is a {\it singular Kato surface}\, meaning that $S$ and its minimal good
resolution $\widetilde{S}$ satisfying all of the above conditions.

\begin{figure}[ht!]
\includegraphics{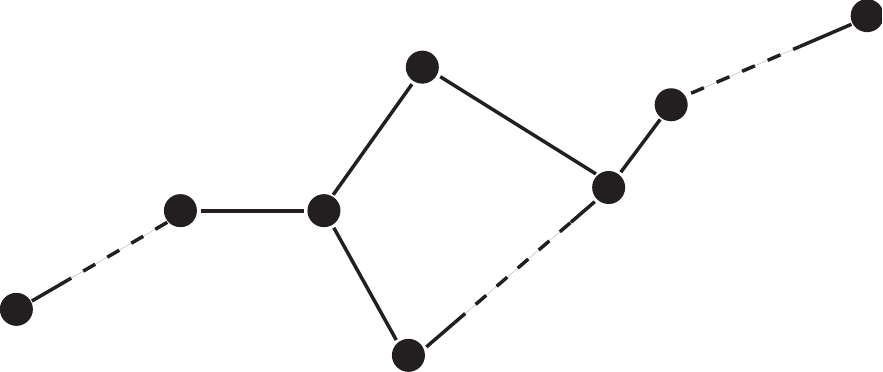}
\caption{Divisor $\widetilde{D}$}\label{uniquefigure}
\end{figure}

On the other hand, it also follows from Lemma~\ref{finishingjob-2.2} that the fiber $F^{-1} (t)$ over $t \in \mathbb{D}^{\ast}$
is a complex surface
equipped with a non-singular holomorphic vector fields $Y_{\vert F^{-1} (t)}$.
According to Mizuhara \cite{Holvf}, see also \cite{fred}, the fiber $F^{-1} (t)$ belongs to the following list:
\begin{enumerate}
  \item A complex torus $\C^2 / \Lambda$;

  \item A flat holomorphic fiber bundle over an elliptic curve;

  \item An elliptic surface without singular fibers or with singular fibers of type
  $mI_0$ only. In other words, the singular fibers are elliptic curve with finite multiplicity;

  \item A Hopf surface;

  \item A positive Inoue surface.
\end{enumerate}
Additional information on each of the possibilities above will be given as they become necessary. Recall that the proof
of Theorem~A is reduced to showing that a fibration $F : U \rightarrow \mathbb{D} \subset \C$ satisfying all the preceding
conditions cannot exist. More precisely, assuming that $S =F^{-1} (0)$ is a singular Kato surface, we are going to show
that the regular fibers $F^{-1} (t)$, $t \neq 0$, cannot belong to the Mizuhara's list (1)---(5) of surfaces equipped with non-singular
vector fields. It is, however, convenient to argue by contradiction. Thus,
we assume aiming at a contradiction that $F : U \rightarrow \mathbb{D} \subset \C$
is such that $S =F^{-1} (0)$ is a singular Kato surface and that the remaining fibers $F^{-1} (t)$ are one of the surfaces in
Mizuhara's list.

Next consider a small open ball $B_{\varepsilon}^6 \subset U$ around $p \in S = F^{-1} (0)$ whose boundary will be denoted by
$\Delta_{\varepsilon}^5$. The ball $B_{\varepsilon}^6$ can be chosen so that its boundary $\Delta_{\varepsilon}^5$
intersects $S$ transversely.
We set $S_{\varepsilon} = B_{\varepsilon}^6 \cap S$ and $S_c = S \setminus (B_{\varepsilon}^6 \cup \Delta_{\varepsilon}^5)$. Thus
both sets $S_{\varepsilon}$ and $S_c$ are open in $S$ and they have the same boundary which coincides with
$S \cap \Delta_{\varepsilon}^5$. Now, we also set $S_{\Delta} = S \cap \Delta_{\varepsilon}^5 = \partial S_{\varepsilon}
= \partial S_c$.

Next note that $S_{\Delta} = S \cap \Delta_{\varepsilon}^5$ is a real $3$-dimensional manifold whose first Betti number equals~$1$
as it follows from the fact that the minimal resolution $\widetilde{S}$ is a Kato surface (cf. the structure of the
divisor $\widetilde{D}$).

Now consider a regular fiber $F_t = F^{-1} (t)$, $t \neq 0$, of $F : U \rightarrow \mathbb{D} \subset \C$. Modulo
choosing $t$ very small, we set $F_{t, \varepsilon} = B_{\varepsilon}^6 \cap F_t$ so that in $B_{\varepsilon}^6$ we obtain the
classical situation associated with Milnor's fibration theorem \cite{milnor}. In particular the surfaces
$F_{t, \varepsilon}$ are connected and simply connected; see \cite{milnor}. Next we denote by
$F_{t,\Delta}$ the intersection of $F_t$ with $\Delta_{\varepsilon}^5$. Finally $F_{t,c}$ will denote the open
set of $F_t$ given by $F_{t,c} = F_t \setminus (F_{t, \varepsilon} \cup F_{t,\Delta})$.

\begin{lemma}
\label{finishingjob-3.3}
The first Betti number of the fiber $F_t = F^{-1} (t)$ is at most one provided that $t\neq 0$ is small enough.
\end{lemma}

\begin{proof}
First we recall that the first Betti number of a Kato surface is equal to~$1$. Now, by
considering the fibration induced by $F$ as a foliation, it is clear that the holonomy associated to the leaf
$S\setminus \{ p\}$ is finite. Thus, the manifold with boundary $F_{t,c} \cup F_{t,\Delta}$ is a finite covering
of $S_c \cup S_{\Delta}$.

The preceding implies that the first Betti number of $F_{t,\Delta}$ equals the first Betti number of
$S_{\Delta}$, namely~$1$. Similarly, the first Betti number of $F_{t,c}$ equals the first Betti number of
$S_c$. Since, in turn, $F_{t, \epsilon}$ is simply connected, the lemma follows from a simple
application of Mayer-Vietoris argument.
\end{proof}

Lemma~\ref{finishingjob-3.3} implies that the fibers $F_t$, $t\neq 0$, cannot be as in items~(1)---(3) of
Mizuhara's list. Thus, in order to finish the proof of Theorem~A, we just need to investigate the case
where the fibers $F_t$ are Hopf surfaces and the case where they are positive Inoue surfaces.

First recall that a {\it Hopf surface}\, is a surface obtained as the quotient of $\C^2 \setminus \{ (0,0) \}$ by
a free action of a discrete group. Similarly, a {\it positive Inoue surface}\, is obtained as the quotient
of $\mathbb{D} \times \C$ by a discrete, cocompact group of automorphisms having the form
$(w,z) \mapsto (\gamma (w) , z + a (w))$ where $\mathbb{D}$ denotes the hyperbolic disc, $\gamma$ is an isometry
of $\mathbb{D}$, and $a$ is a holomorphic function on $\mathbb{D}$. The reader will note that
all these surfaces have second Betti number equal to zero.

\begin{proof}[Proof of Theorem A]
Consider again the fiber $F_t$ and its subsets $F_{t, \epsilon}$, $F_{t,\Delta}$, and $F_{t,c}$. Note also
that the second homology group of $F_{t, \epsilon}$ is generated by the {\it vanishing cycles}\, and it is free abelian
on at least one generator; see \cite{milnor}. Since the second Betti number of $F_t$ vanishes, the Mayer-Vietoris
sequence yields
$$
H_3 (F_t) \longrightarrow H_2 (F_{t,\Delta}) \stackrel{(i_{\ast}, j_{\ast})}\longrightarrow H_2 (F_{t, \epsilon}) \oplus H_2 (F_{t,c})
\longrightarrow 0 \, .
$$
Note that the second Betti number of $F_{t,\Delta}$ is one by Poincar\'e duality. On the other
hand, the homomorphism $(i_{\ast}, j_{\ast})$ is onto $H_2 (F_{t, \epsilon}) \oplus H_2 (F_{t,c})$. Hence
the rank of $H_2 (F_{t, \epsilon}) \oplus H_2 (F_{t,c})$ must be~$1$ since it is necessarily strictly positive.
In other words, the {\it Milnor number}\, of the singular surface $F^{-1} (0)$ is one; see \cite{milnor}.

To derive a final contradiction proving Theorem~A, it suffices to check that the Milnor number $\mu$ associated
with the singular point $p \in S=F^{-1} (0)$ must be strictly greater than one. A short argument in this direction
requires Laufer's formula in \cite{laufer} stating that
\begin{equation}
1+ \mu = \mathcal{E}_{\rm top} \, (\widetilde{D}) + K.K + 12 \, {\rm dim}_{\C} \, H^1 (\widetilde{S} , \mathcal{O}) \, , \label{laufer}
\end{equation}
where $\mathcal{E}_{\rm top} \, (\widetilde{D})$ stands for the topological Euler characteristic of the exceptional
divisor $\widetilde{D}$ and where $K.K$ is the self-intersection of the canonical class in the resolution $\widetilde{S}$
of $S$. On the other hand, recalling that the {\it holomorphic Euler characteristic}\, $\mathcal{X} (\widetilde{S} )$
of $\widetilde{S}$ is defined by $\mathcal{X} (\widetilde{S} ) = 1 - {\rm dim}_{\C} \, H^1 (\widetilde{S} , \mathcal{O})
+{\rm dim}_{\C} \, H^2 (\widetilde{S} , \mathcal{O})$, Noether formula yields
$$
K.K = 12 \, \mathcal{X} \, (\widetilde{S} ) - \mathcal{E}\, (\widetilde{S})
$$
where $\mathcal{E}\, (\widetilde{S})$ is the usual Euler characteristic of the real four dimensional manifold
$\widetilde{S}$; see \cite{friedman}, pages~8 and~9. There follows that
$$
K.K + 12 \, {\rm dim}_{\C} \, H^1 (\widetilde{S} , \mathcal{O}) = 12 + 12 \, {\rm dim}_{\C} \, H^2 (\widetilde{S} , \mathcal{O})
- \mathcal{E} \, (\widetilde{S}) \, .
$$
Combined to Formula~(\ref{laufer}), the preceding equation yields
$$
1+ \mu = 12 + 12 \, {\rm dim}_{\C} \, H^2 (\widetilde{S} , \mathcal{O}) + \mathcal{E}_{\rm top} \, (\widetilde{D}) -
\mathcal{E} \, (\widetilde{S}) \, .
$$
Thus, to finish the proof of Theorem~A it is enough to check that $\mathcal{E}_{\rm top} \, (\widetilde{D}) \geq
\mathcal{E} \, (\widetilde{S})$. This is however easy: since the first Betti number of $\widetilde{S}$ equals~$1$,
Poincar\'e duality ensures that $\mathcal{E} \, (\widetilde{S})$ is equal to the second Betti number of $\widetilde{S}$.
In turn, this second Betti number is nothing but the number of irreducible components in the exceptional divisor
$\widetilde{D}$; see \cite{adolfo}. Again, the description of the exceptional divisor $\widetilde{D}$ shows
that $\mathcal{E}_{\rm top} (\widetilde{D})$ also equals the number of connected components in $\widetilde{D}$. The
proof of Theorem~A is completed.
\end{proof}

\bigskip

\noindent {\bf Acknowledgements}. The presentation of this paper benefited greatly from comments and
questions by A. Lins-Neto, J.V. Pereira and P. Sad to whom we are mostly grateful. We also thank
F. Bosio for providing us with accurate information about complex surfaces and non-singular vector fields.

This research was supported by FCT through project EXPL/MAT-CAL/1575/2013.
The second author was also partially supported by FCT through CMUP.

\bigskip

\begin{flushleft}
{\sc Julio Rebelo} \\
Institut de Math\'ematiques de Toulouse ; UMR 5219\\
Universit\'e de Toulouse\\
118 Route de Narbonne\\
F-31062 Toulouse, FRANCE.\\
rebelo@math.univ-toulouse.fr

\end{flushleft}

\bigskip

\begin{flushleft}
{\sc Helena Reis} \\
Centro de Matem\'atica da Universidade do Porto, \\
Faculdade de Economia da Universidade do Porto, \\
Portugal\\
hreis@fep.up.pt \\

\end{flushleft}

\end{document}